\documentclass{amsart}
\catcode`\é=\active
\defé{\'e}

\catcode`\è=\active
\defè{\`e}

\catcode`\ê=\active
\defê{\^e}

\catcode`\û=\active
\defû{\^u}

\catcode`\â=\active
\defâ{\^a}

\catcode`\à=\active
\defà{\`a}

\catcode`\ù=\active
\defù{\`u}

\catcode`\ç=\active
\defç{\c c}

\catcode`\î=\active
\defî{\^{\i}}

\catcode`\ï=\active
\defï{\"{\i}}

\catcode`\ô=\active
\defô{\^o}

\catcode`\é=\active
\defé{\'e}

\def\zb{{\Bbb Z}}

\usepackage{amssymb}
\usepackage{url}
\usepackage{amscd}

\newtheorem{theo}{Theorem}[section]
\newtheorem{lem}[theo]{Lemma}
\newtheorem{prop}[theo]{Proposition}

\theoremstyle{definition}

\newtheorem{rema}[theo]{Remark}
\numberwithin{equation}{section}

\begin{document}

\title[The saddle-point method and the Li coefficients]{The saddle-point method and the\\ Li coefficients}
\author{Kamel Mazhouda}
\address{Faculté des sciences de Monastir, Département de mathématiques, 5000 Monastir, Tunisia}
\email{kamel.mazhouda@fsm.rnu.tn}
%

\begin{abstract}
In this paper, we apply the saddle-point method in conjunction with the theory of the N$\ddot{o}$rlund-Rice integrals to derive a precise asymptotic formula for the generalized Li coefficients established  by Omar and Mazhouda. Actually, for any function $F$ in the Selberg class $\mathcal{S}$ and under the Generalized Riemann Hypothesis, we have
$$\lambda_{F}(n)=\frac{d_{F}}{2}n\log n+c_{F}n+O(\sqrt{n}\log n),$$
with $$c_{F}=\frac{d_{F}}{2}(\gamma-1)+\frac{1}{2}\log(\lambda Q_{F}^{2}),\ \ \lambda=\prod_{j=1}^{r}\lambda_{j}^{2\lambda_{j}},$$
where $\gamma$ is the Euler constant and  the notation is as bellow.
\end{abstract}

\maketitle


\section{Introduction}
\label{sec:introduction}

Let us consider the xi-function
 $\xi(s)=s(s-1)\Gamma(s/2)\pi^{-s/2}\zeta(s)$ and the Li coefficients
 $(\lambda_n)_{n\geq 1}$ defined by
$$\lambda_{n}=\frac{1}{(n-1)!}\frac{d^{n}}{ds^{n}}\left[s^{n-1}\log\xi(s)\right]_{s=1}.$$
Then, the Li criterion says that the Riemann Hypothesis holds if and only if
the coefficients $(\lambda_n)$ are positive numbers.
 Bombieri and Lagarias \cite{2} obtained an arithmetic expression for the Li coefficients $\lambda_{n}$ and gave an asymptotic formula as $n\rightarrow\infty$. More recently, Maslanka \cite{10} computed $\lambda_{n}$ for $1\leq n\leq3300$ and empirically studied the growth behavior of the Li coefficients. Coffey \cite{3,4} studied the arithmetic formula and established a lower bounded for  the Archimedean prime contribution by means of series rearrangements using the Euler-Maclaurin summation. In \cite{11}, a generalization of  the Li criterion for  functions $F$ in the Selberg class was given, and in \cite{13} an explicit formula for the Li coefficients associated to $F$ was established. \\

The object of this paper is to derive a precise asymptotic formula for the generalized Li coefficients using the saddle-point method.\\

The Selberg class ${\mathcal S}$ consists of Dirichlet series
$$F(s)=\sum_{n=1}^{+\infty}\frac{a(n)}{n^{s}}, \qquad \Re(s)>1 $$
satisfying the following hypotheses.\\
\begin{itemize}
    \item {\bf   Analytic continuation:} there exists a non negative integer $m$ such $(s-1)^{m}F(s)$ is an entire function of finite order.  We denote by  $m_{F}$  the smallest integer  $m$  which satisfies this condition.
    \item {\bf  Functional equation:} for $1\leq j\leq r$, there are positive real numbers $Q_{F},\ \lambda_{j}$ and there are complex numbers $\mu_{j},\ \omega$ with $\Re(\mu_{j})\geq0$ and $|\omega|=1$, such that
       $$\phi_{F}(s)=\omega\overline{\phi_{F}(1-\overline{s})}$$
    where
    $$\phi_{F}(s)=F(s)Q_{F}^{s}\prod_{j=1}^{r}\Gamma(\lambda_{j}s+\mu_{j}).$$
    \item {\bf Ramanujan hypothesis:} $a(n)=O(n^{\epsilon})$.
    \item {\bf Euler product:} $F(s)$ satisfies
    $$F(s)=\prod_{p}\exp\left(\sum_{k=1}^{+\infty}\frac{b(p^{k})}{p^{ks}}\right)$$
\noindent with suitable coefficients $b(p^{k})$ satisfying $b(p^{k})=O(p^{k\theta})$ for some $\theta>\frac{1}{2}$.
\end{itemize}
 It is  expected that for every function in the Selberg class the analogue of the Riemann hypothesis holds, i.e, that all non trivial (non-real) zeros lie on the critical line $\Re(s)=\frac{1}{2}$.
The degree of $F\in{{\mathcal S}}$ is defined by
$$d_{F}=2\sum_{j=1}^{r}\lambda_{j}.$$
The degree is well defined (although the functional equation is not
unique by the Legendre's duplication formula). The logarithmic
derivative of $F(s)$ also has the Dirichlet series expression
$$-\frac{F'}{F}(s)=\sum_{n=1}^{+\infty}\Lambda_{F}(n)n^{-s},\qquad \Re(s)>1,$$
where $\Lambda_{F}(n)=b(n)\log n$ is an analogue of the Von Mongoltd
function $\Lambda(n)$ defined by
$$\Lambda(n)=\left\{\begin{array}{crll}\log p\quad  \hbox{if}&  n=p^{k}\  \hbox{with}\
k\geq1,\\
0 & \hbox{otherwise.}\end{array}\right.$$ If $N_{F}(T)$ counts the
number of zeros of $F(s)\in{{\mathcal S}}$ in the rectangle $0\leq
\Re(s)\leq1$,\ $|\Im(s)|\leq T$ (according to multiplicities), one can
show by standard contour integration the formula
$$N_{F}(T)=\frac{d_{F}}{2\pi}T\log T+c_{F}T+O(\log T),$$
in analogy to the Riemann-Von Mangoldt formula for Riemann's zeta-function $\zeta(s)$, the prototype of an element in ${\mathcal S}$. For more details concerning the Selberg class we refer to the survey of Kaczorowski and Perelli \cite{6}.
\section{The Li criterion }
\label{The Li criterion}
Let $F$ be a function in the Selberg class non-vanishing at $s=1$ and let us define the xi-function $\xi_{F}(s)$ by
$\xi_{F}(s)=s^{m_{F}}(s-1)^{m_{F}}\phi_{F}(s).$  The function $\xi_{F}(s)$ satisfies the functional equation
$\xi_{F}(s)=\omega\overline{\xi_{F}(1-\overline{s})}.$ The function $\xi_{F}$ is an entire function of order 1. Therefore by the Hadamard product,
it can be written as
$$
\xi_{F}(s)=\xi_{F}(0)\prod_{\rho}\left(1-\frac{s}{\rho}\right),
$$
where the product is over all zeros of $\xi_{F}(s)$ in the order given by $|\Im(\rho)| < T$ for
$T\rightarrow\infty$. Let $\lambda_{F}(n)$, $n\in{\zb}$, be a sequence of numbers defined by a sum over the non-trivial zeros of $F(s)$ as
$$\lambda_{F}(n)=\sum_{\rho}\left[1-\left(1-\frac{1}{\rho}\right)^{-n}\right],$$
where the sum over $\rho$ is
$$\sum_{\rho}=\lim_{T\mapsto\infty}\sum_{|\Im\rho|\leq T}.$$
These coefficients are expressible in terms of power-series coefficients of functions
constructed from the $\xi_{F}$-function. For $n\leq -1$, the Li coefficients $\lambda_{F}(n)$ correspond
to the following Taylor expansion at the point $s = 1$
$$\frac{d}{dz}\log\xi_{F}\left(\frac{1}{1-z}\right)=\sum_{n=0}^{+\infty}\lambda_{F}(n+1)z^{n}.$$
and for $n \geq1$, they correspond to the Taylor expansion at $s = 0$
$$\frac{d}{dz}\log\xi_{F}\left(\frac{-z}{1-z}\right)=\sum_{n=0}^{+\infty}\lambda_{F}(-n-1)z^{n},$$
Let $Z$ the multi-set of zeros of $\xi_{F}(s)$ (counted with multiplicity). The multi-set $Z$ is invariant under the map $\rho\longmapsto1-\overline{\rho}$, We have
$$1-\left(1-\frac{1}{\rho}\right)^{-n}=1-\left(\frac{\rho-1}{\rho}\right)^{-n}=1-\left(\frac{-\rho}{1-\rho}\right)^{n}=1-\overline{\left(1-\frac{1}{1-\overline{\rho}}\right)^{n}}$$
and this gives the symmetry $\lambda_{F}(-n)=\overline{\lambda_{F}(n)}$. Using the corollary in \cite[Theorem 1]{2}, we get the following generalization  of the Li criterion for the Riemann hypothesis.
\begin{theo} Let $F(s)$ be a function in the Selberg class ${\mathcal S}$ non-vanishing at $s=1$. All non-trivial zeros of $F(s)$ lie in the line $\Re e(s)=1/2$ if and only if $\Re\left(\lambda_{F}(n)\right)>0$ for $n=1,2..$.
\end{theo}
Next, we recall the following explicit formula for the coefficients $\lambda_ {F}(n)$. Let consider
the following hypothesis:\\
${\mathcal H}$ : there exist a constant $c > 0$ such that $F(s)$ is non-vanishing in the region:
$$\left\{s=\sigma+it;\ \sigma\geq1-\frac{c}{\log(Q_{F}+1+|t|)}\right\}.$$

\begin{theo}\label{th2.2} Let $F(s)$ be  a function in the Selberg class ${\mathcal S}$ satisfying ${\mathcal H}$. Then, we have
\begin{eqnarray}\label{eq2.1}
\lambda_{F}(-n)&=&m_{F}+n(\log Q_{F}-\frac{d_{F}}{2}\gamma)\nonumber\\
&-&\sum_{l=1}^{n}(_{l}^{n})\frac{(-1)^{l-1}}{(l-1)!}\ \lim_{X\longrightarrow+\infty}\left\{\sum_{k\leq X}\frac{\Lambda_{F}(k)}{k}(\log k)^{l-1}-\frac{m_{F}}{l}(\log X)^{l}\right\}\nonumber\\
&+&n\sum_{j=1}^{r}\lambda_{j}\left(-\frac{1}{\lambda_{j}+\mu_{j}}+\sum_{l=1}^{+\infty}\frac{\lambda_{j}+\mu_{j}}{l(l+\lambda_{j}+\mu_{j})}\right)\nonumber\\
&-&\sum_{j=1}^{r}\sum_{k=2}^{n}(_{k}^{n})(-\lambda_{j})^{k}\sum_{l=0}^{+\infty}\left(\frac{1}{l+\lambda_{j}+\mu_{j}}\right)^{k},
\end{eqnarray}
where $\gamma$ is the Euler constant.
 \end{theo}
 \noindent {\bf Examples:}\\
- In the case of the Riemann
zeta function, $m_{\zeta}=1,\ Q_{\zeta}=\pi^{-1/2}, r=1,\
\lambda_{1}=1/2$, and $\mu_{1}=0$. With the equality
$$(-1)^{k}\sum_{l=0}^{+\infty}\left(\frac{1}{2l+1}\right)^{k}=(-1)^{k}\left(1-\frac{1}{2^{k}}\right)\zeta(k),$$
we find  $\lambda_{\zeta}$ which was  established by Bombieri and Lagarias \cite[Page 281]{2}.\\
- For the  Hecke $L$-functions,  $Q_{F}=\frac{\sqrt{N}}{2\pi},\ m_{F}=0,\  \lambda_{1}=1$, and
$\mu_{1}=\frac{1}{2}$, we find  $\lambda_{E}(n)$, which was established by X. -J. Li \cite[p. 496]{9}.
\section{Saddle-point method and the N$\ddot{o}$rlund-Rice integrals }
\label{sec:Saddle-point method }
Given a complex integral with a contour traversing simple saddle-point, the saddle-point corresponds locally to a maximum of the integrand along the path. It is then natural to expect that a small neighborhood of the saddle-point might provide the dominant contribution to the integral. The saddle-point method is applicable precisely when this is the case and when this dominant contribution can be estimated by means of local expansions. The method then constitutes the complex-analytic counterpart of the method of Laplace for evaluation of real integrals depending on a large parameter, and we can regard it as being
\begin{center}
Saddle-point method=Choice of contour + Laplace's method.
\end{center}
\noindent To estimate $\int_{A}^{B}F(z)dz$, it is convenient to set $F(z)=e^{f(z)}$ where $f(z)\equiv f_{n}(z)$, involves some large parameter $n$. We chose a contour $\mathcal{C}$ through a saddle-point $\eta$ such that $f'(\eta)=0$. Next, we split the contour as $\mathcal{C}=\mathcal{C}^{(0)}\cup\mathcal{C}^{(1)}$, and the following conditions are to be verified.\\
i) On the contour $\mathcal{C}^{(1)}$ the tails integral$\int_{\mathcal{C}^{(1)}}$ is negligible
$$\int_{\mathcal{C}^{(1)}}F(z)dz=O\left(\int_{\mathcal{C}}F(z)dz\right).$$
ii) Along $\mathcal{C}^{(0)}$, a quadratic expansion,
$$f(z)=f(\eta)+\frac{1}{2}f''(\eta)(z-\eta)^{2}+O(\phi_{n})$$
is valid, with $\phi_{n}\rightarrow0$ as $n\rightarrow\infty$, uniformly with respect to $z\in{\mathcal{C}^{(0)}}$.\\
ii) The incomplete Gaussien integral taken over the central range is asymptotically equivalent to a complete Gaussien integral with $\epsilon=\pm1$ and $f''(\eta)=e^{i\varphi}|f''(\eta)|$:
$$\int_{\mathcal{C}^{(0)}}e^{\frac{1}{2}f''(\eta)(z-\eta)^{2}}dz\sim\epsilon i e^{-i\varphi/2}\int_{-\infty}^{+\infty}e^{-|f''(\eta)|\frac{x^{2}}{2}}dx\equiv\epsilon i e^{-i\varphi/2}\sqrt{\frac{2\pi}{|f''(\eta)|}}.$$
Assuming i), ii) and iii), one has, with $\epsilon=\pm1$ and $\arg(f''(\eta))=\varphi$
$$\frac{1}{2\pi}\int_{A}^{B}e^{f(z)}dz\sim\epsilon e^{-i\varphi/2}\sqrt{\frac{2\pi}{|f''(\eta)|}}e^{f(\eta)}=\pm \sqrt{\frac{2\pi}{|f''(\eta)|}}e^{f(\eta)}.$$
This method is the main tool to prove our result.
We finish this section by reviewing the definition of the N$\ddot{o}$rlund-Rice integral.
\begin{lem}\label{lem3.1}
Let $f(s)$ be holomorphic in the half-plane $\Re(s)\geq\eta_{0}-\frac{1}{2}$. Then the finite differences of the sequence $(f(k))$ admit the integral representation
$$\sum_{k=n_{0}}^{n}\left(_{k}^{n}\right)(-1)^{k}f(k)=\frac{(-1)^{n}}{2i\pi}\int_{\mathcal{C}}f(s)\frac{n!}{s(s-1)....(s-n)}ds,$$
where the contour of integration $\mathcal{C}$ encircles the integers $\{n_{0},..,n\}$ in a positive direction and is contained in $\Re(s)\geq\eta_{0}-\frac{1}{2}$.
\end{lem}
\begin{proof} The integral on the right is the sum of its residues at $s=n_{0},..,n$, which precisely equals the sum on the left.
\end{proof}
\section{Asymptotic formula for the Li coefficients}
\label{sec:Asymptotic formula for the Li coefficients}
A natural question is to determine the asymptotic behavior of the numbers $\lambda_{F}(n)$. Our main result in this paper is stated in the following theorem.
\begin{theo}\label{th4.1}
Let $F(s)$  be a function in the Selberg class ${\mathcal S}$. Then, under the Generalized Riemann Hypothesis, we have
$$\lambda_{F}(n)=\frac{d_{F}}{2}n\log n+c_{F}n+O\left(\sqrt{n}\log n\right),$$
where $$c_{F}=\frac{d_{F}}{2}(\gamma-1)+\frac{1}{2}\log(\lambda Q_{F}^{2}),\ \ \lambda=\prod_{j=1}^{r}\lambda_{j}^{2\lambda_{j}}$$
and $\gamma$ is the Euler constant.
\end{theo}
\begin{rema}
We conjecture that the asymptotic formula for the numbers $\lambda_{F}(n)$ in  Theorem  \ref{th4.1} holds for any function in the Selberg class without any assumption.
\end{rema}
 Four our purpose, it is sufficient to study sums of the form
\begin{equation}\label{eq4.1}
H_{n}(m,k)=\sum_{l=2}^{n}(-1)^{l}\left(_{l}^{n}\right)\frac{\zeta(l,\frac{m}{k})}{k^{l}},
\end{equation}
where $m$ and $k$ are tow complex numbers, $\zeta(s,q)$ is the Hurwitz zeta function and defined for complex arguments $s$ with $\Re(s) > 1$ and $q$ with $\Re(q) > 0$ by
$$\zeta(s,q)=\sum_{n=0}^{+\infty}\frac{1}{(n+q)^{s}}.$$
\begin{prop}\label{pro4.3}
 $H_{n}(m,k)$, defined by \eqref{eq4.1}, satisfy the estimate
$$ H_{n}(m,k)=\left(\frac{m}{k}-\frac{1}{2}\right)-\frac{n}{k}\left(\psi(\frac{m}{k})+\log k+1-h_{n-1}\right)+a_{n}(m,k),$$
where the $a_{n}(m,k)$ are exponentially small:
$$
a_{n}(m,k)= \frac{1}{k}\left(\frac{2n}{\pi k}\right)^{1/4}\exp(-\sqrt{\frac{4\pi n}{k}})\cos\left(\sqrt{\frac{4\pi n}{k}}-\frac{5\pi}{8}-\frac{2\pi m}{k}\right)
+O\left(n^{-1/4}e^{-2\sqrt{\frac{\pi n}{k}}}\right).$$
Here, $h_{n}=1+\frac{1}{2}+..+\frac{1}{n}$ is a harmonic number, and $\psi(x)$ is the logarithm derivative of the Gamma function.
\end{prop}
The previous results for the  Riemann zeta function may be regained by setting $m = k = 1$ and for  the Dirichlet
$L$-functions $m = 1$ and $k = 2$.
\begin{proof}
Convert the sum to the N$\ddot{o}$rlund-Rice integral, and extending the contour to the half-circle at positive infinity. The half-circle does not contribute to the integral. One obtains
$$H_{n}(m,k)=\frac{(-1)^{n}}{2i\pi}n!\int_{3/2-i\infty}^{3/2+i\infty}\frac{\zeta(s,\frac{m}{k})}{k^{l}s(s-1)..(s-n)}ds.$$
Moving the integral to the left, one encounters single pole at $s=0$ and a pole at $s=1$. The residue of the pole at $s=0$ is
$$Res(s=0)=\zeta(0,\frac{m}{k})=-\frac{1}{k\pi}\sum_{l=1}^{k}\sin\left(\frac{2\pi lm}{k}\right)\psi\left(\frac{l}{k}\right)=-B_{1}\left(\frac{l}{k}\right)=\frac{1}{2}-\frac{m}{k},$$
where $\psi$ is the digamma function, $B_{1}$ is the Bernoulli polynomial of order 1, and
$$Res(s=1)=\frac{n}{k}\left(\psi(\frac{m}{k})+\log k+1-h_{n-1}\right).$$
Then we obtain
$$H_{n}(m,k)=\left(\frac{m}{k}-\frac{1}{2}\right)-\frac{n}{k}\left(\psi(\frac{m}{k})+\log k+1-h_{n-1}\right)+a_{n}(m,k),
$$
where $$a_{n}(m,k)=O\left(e^{-\sqrt{Kn}}\right),$$
for a constant $K$ of order $m/k$. Indeed, we have
$$a_{n}(m,k)=\frac{(-1)^{n}}{2i\pi}n!\int_{-1/2-i\infty}^{-1/2+i\infty}\frac{\zeta(s,\frac{m}{k})}{k^{l}s(s-1)..(s-n)}.$$
Recall that the Hurwitz zeta function satisfies the following functional equation
$$\zeta\left(1-s,\frac{m}{k}\right)=\frac{2\Gamma(s)}{(k\pi k)^{s}}\sum_{l=1}^{k}\cos\left(\frac{\pi s}{2}-\frac{2\pi lm}{k}\right)\zeta\left(s,\frac{l}{k}\right).$$
Therefore,
\begin{eqnarray}\label{eq4.2}
a_{n}(m,k)&=&-\frac{n!}{2ki\pi}\sum_{l=1}^{k}\int_{3/2-i\infty}^{3/2+i\infty}\frac{1}{(2\pi)^{s}}\frac{\Gamma(s)\Gamma(s-1)}{\Gamma(s+n)}\cos\left(\frac{\pi s}{2}-\frac{2\pi lm}{k}\right)\zeta\left(s,\frac{l}{k}\right)ds\nonumber\\
&=&-\frac{n!}{2ki\pi}\sum_{l=1}^{k}e^{i\frac{2\pi lm}{k}}\int_{3/2-i\infty}^{3/2+i\infty}\frac{1}{(2\pi)^{s}}\frac{\Gamma(s)\Gamma(s-1)}{\Gamma(s+n)}e^{-i\frac{\pi s}{2}}\zeta\left(s,\frac{l}{k}\right)ds\\
&&-\frac{n!}{2ki\pi}\sum_{l=1}^{k}e^{-i\frac{2\pi lm}{k}}\int_{3/2-i\infty}^{3/2+i\infty}\frac{1}{(2\pi)^{s}}\frac{\Gamma(s)\Gamma(s-1)}{\Gamma(s+n)}e^{i\frac{\pi s}{2}}\zeta\left(s,\frac{l}{k}\right)ds.\nonumber
\end{eqnarray}
For large values of $n$, those integrals will be evaluated by means of the saddle-point method. Noting that the integrand in \eqref{eq4.2} have a minimum, on the real axis, near $s=\sigma_{0}=\sqrt{\frac{2ln}{k}}$, and so the appropriate parameter is $z=s/\sqrt{n}$. Changing $s$ by $z$, and take $z$ constant and $n$ large. Then
\begin{equation}\label{eq4.3}
a_{n}(m,k)=-\frac{1}{2i\pi}\sum_{l=1}{k}\left\{e^{i\frac{2\pi lm}{k}}\int_{\sigma_{0}-i\infty}^{\sigma_{0}+i\infty}e^{f(z)}dz+e^{-i\frac{2\pi lm}{k}}\int_{\sigma_{0}-i\infty}^{\sigma_{0}+i\infty}e^{\overline{f}(z)}dz\right\}.
\end{equation}
We have
$$f(z)=\log n!+\frac{1}{2}\log n+\phi(z\sqrt{n}),$$
with
$$\phi(s)=-s\log\left(\frac{2\pi l}{k}\right)-i\frac{\pi s}{2}+\log\left(\frac{\Gamma(s)\Gamma(s-1)}{\Gamma(s+n)}\right)+O\left(\left(\frac{l}{k+l}\right)^{s}\right),$$
using the approximation
$$\zeta(s,l/k)=(k/l)^{s}+O\left(\left(\frac{l}{k+l}\right)^{s}\right)$$
for larg $s$. Furthermore,
$$\log\zeta(s)=\sum_{n=2}^{+\infty}\frac{\Lambda(n)}{n^{s}\log n},$$
where $\Lambda(n)$ is the Von-Mangoldt function. The asymptotic expansion for the Gamma function is given by the following Stirling expansion
$$\log\Gamma(x)=(x-\frac{1}{2})\log x-x+\frac{1}{2}\log(2\pi)+\sum_{j=1}^{+\infty}\frac{B_{2j}}{2j(2j-1)x^{2j-1}},$$
where $B_{k}$ are the Bernoulli numbers. Expanding to $O(1/n)$ and collecting terms, we deduce
\begin{eqnarray}
f(z)&=&\frac{1}{2}\log n-z\sqrt{n}\left(\log\left(\frac{2\pi l}{k}\right)+i\frac{\pi}{2}+2-2\log z\right)
+\log(2\pi)\nonumber\\&&-2\log z-\frac{z^{2}}{2}+\frac{1}{6z\sqrt{n}}\left(10+z^{2}\right)\nonumber\\
&&+\frac{1}{2n}\left(1-\frac{z^{2}}{2}-\frac{z^{4}}{6}+\frac{73}{72z^{2}}\right)+O\left(n^{-1/2}\right).\nonumber
\end{eqnarray}
The saddle-point is obtained by solving the equation $f'(z)=0$, and we have $$z_{0}=(1+i)\sqrt{\frac{\pi l}{k}}.$$
Let $\sigma\equiv\sigma(n)=z_{0}\sqrt{n}=(1 + i)\sqrt{\frac{\pi ln}{k}}=\sqrt{\frac{2\pi ln}{k}}e^{i\pi/4}$, which thus is also an approximate saddle-point for $e^{f(z)}$. Now, we fix the contour of integration and provide final approximations. The
contour  goes through the saddle point $\sigma=\sigma(n)$ and symmetrically
through its complex conjugate $\overline{\sigma}=\overline{\sigma(n)}$. In the upper half-plane, it traverses $\sigma(n)$ along
a line of steepest descent whose direction, as determined from the argument of $f''(\sigma)$. The contour also includes parts of two vertical
lines of respective abscissae $Re(s) = c_{1}\sqrt{n}$ and $c_{2}\sqrt{n}$, where
$0 < c_{1} < \sqrt{\frac{2\pi ln}{k}}< c_{2} < 2\sqrt{\frac{2\pi ln}{k}}$.
The choice of the abscissae, $c_{1}$ and $c_{2}$, is not critical (it is even possible to adapt the
analysis to $c_{1} = c_{2} = \sqrt{\frac{\pi l}{k}}$).
\begin{itemize}
  \item Contribution of the vertical parts of the contour are $O\left(e^{-L_{0}\sqrt{n}}\right)$, for some $L_{0}>2\sqrt{\frac{2\pi ln}{k}}.$
  \item The slanted part of the contour is such that all the estimates of $f(z), f'(z)$ and $f''(z)$ apply. The scale of the problem is dictated by the value of $f''(\sigma)$,  which of order $O(n^{-1/2})$. We have
      $$\int_{\hbox{slanted}}=\int_{\hbox{central}}\left(1+O(n^{-1/2})\right).$$
      Using that $z_{0}=\sqrt{\frac{2\pi l}{k}}e^{i\pi/4},\ f(\sqrt{n}z_{0})=-2\sqrt{n}z_{0}$ and $$f''(\sqrt{n}z_{0})=2/(\sqrt{n}z_{0})+O(n^{-1}).$$
      Now, we use  the saddle-point formula to obtain
      $$\int_{\hbox{central}}=e^{f(\sqrt{n}z_{0})}\sqrt{\frac{2\pi}{|f''(\sqrt{n}z_{0})|}}=\left(\frac{2\pi^{3}ln}{k}\right)^{1/4}e^{i\pi/8}e^{\left(-(1+i)\sqrt{\frac{4\pi ln}{k}}\right)}.$$
      Hence
      $$\int_{\hbox{slanted}}=\left(\frac{2\pi^{3}ln}{k}\right)^{1/4}e^{i\pi/8}e^{\left(-(1+i)\sqrt{\frac{4\pi ln}{k}}\right)}+O\left(n^{-1/4}e^{-2\sqrt{\frac{\pi ln}{k}}}\right).$$
\end{itemize}
 Substituting, we obtain
\begin{equation}\label{eq4.4}
\int_{\sigma_{0}-i\infty}^{\sigma_{0}+i\infty}e^{f(z)}dz=\left(\frac{2\pi^{3}ln}{k}\right)^{1/4}e^{\frac{i\pi}{8}}\exp\left(-(1+i)\sqrt{\frac{4\pi ln}{k}}\right)\quad+O\left(n^{-1/4}e^{-2\sqrt{\frac{\pi ln}{k}}}\right).
\end{equation}
The integral for $\overline{f}$ is the complex conjugate of \eqref{eq4.4} (having a saddle-point at the complex conjugate $\overline{z_{0}}$). Finally, equations \eqref{eq4.3} and \eqref{eq4.4} together gives
\begin{eqnarray}
a_{n}(m,k)&=& \frac{1}{k}\left(\frac{2n}{\pi}\right)^{1/4}\sum_{l=1}^{k}\left(\frac{l}{k}\right)^{1/4}\exp(-\sqrt{\frac{4\pi ln}{k}})\ \cos\left(\sqrt{\frac{4\pi ln}{k}}-\frac{5\pi}{8}-\frac{2\pi lm}{k}\right)\nonumber\\
&&\quad+\ O\left(n^{-1/4}e^{-2\sqrt{\frac{\pi ln}{k}}}\right).\nonumber
\end{eqnarray}
For large $n$, only the $l=1$ term contributes significantly, and so
\begin{equation}
a_{n}(m,k)= \frac{1}{k}\left(\frac{2n}{\pi k}\right)^{1/4}\exp(-\sqrt{\frac{4\pi n}{k}})\cos\left(\sqrt{\frac{4\pi n}{k}}-\frac{5\pi}{8}-\frac{2\pi m}{k}\right)+O\left(n^{-1/4}e^{-2\sqrt{\frac{\pi n}{k}}}\right),\nonumber
\end{equation}
which means that the terms $a_{n}$ are exponentially small.
\end{proof}
\noindent {\bf Proof of Theorem \ref{th4.1}} Since, we interest to the real part of $\lambda_{F}(n)$, then, without loss of generality, we assume that $\mu_{j}$ is a real number. First, write the arithmetic formula of $\lambda_{F}(-n)$ (equation \eqref{eq2.1}) as
\begin{eqnarray}\label{eq4.5}
\lambda_{F}(-n)&=&m_{F}+n(\log Q_{F}-\frac{d_{F}}{2}\gamma)-\sum_{l=1}^{n}(_{l}^{n})\eta_{F}(l-1)\nonumber\\
&+&n\sum_{j=1}^{r}\lambda_{j}\left(-\frac{1}{\lambda_{j}+\mu_{j}}+\sum_{l=1}^{+\infty}\frac{\lambda_{j}+\mu_{j}}{l(l+\lambda_{j}+\mu_{j})}\right)-\sum_{j=1}^{r}I_{j},
\end{eqnarray}
where $$\eta_{F}(l)=\frac{(-1)^{l}}{l!}\ \lim_{X\longrightarrow+\infty}\left\{\sum_{k\leq X}\frac{\Lambda_{F}(k)}{k}(\log k)^{l}-\frac{m_{F}}{l+1}(\log X)^{l+1}\right\}$$
are the generalized Stieltjes constants and
$$I_{j}=\sum_{k=2}^{n}(_{k}^{n})(-\lambda_{j})^{k}\sum_{l=1}^{+\infty}\left(\frac{1}{l+\lambda_{j}+\mu_{j}}\right)^{k}.$$
Note that
$$I_{j}=\sum_{k=2}^{n}(_{k}^{n})(-\lambda_{j})^{k}\sum_{l=0}^{+\infty}\frac{1}{(l+\lambda_{j}+\mu_{j})^{k}}=\sum_{k=2}^{n}(_{k}^{n})(-1)^{k}\frac{\zeta(k,\lambda_{j}+\mu_{j})}{(\lambda_{j}^{-1})^{k}},$$
which, with the above notation of $H_{n}(m,k)$ (equation \eqref{eq4.1}), is equal to
$$I_{j}=H_{n}(1+\frac{\mu_{j}}{\lambda_{j}},\lambda_{j}^{-1}).$$
Since $\lambda_{j}>0$ and $\mu_{j}=Re(\mu_{j})\geq0$, applying Proposition \ref{pro4.3} with $m=1+\frac{\mu_{j}}{\lambda_{j}}$ and $ k=\lambda_{j}^{-1}$, we deduce
\begin{equation}\label{eq4.6}
I_{j}=\left(\lambda_{j}+\mu_{j}-\frac{1}{2}\right)-n\lambda_{j}\left(\psi(\lambda_{j}+\mu_{j})+\log(\lambda_{j}^{-1})+1-h_{n-1}\right)+a_{n}(1+\frac{\mu_{j}}{\lambda_{j}},\lambda_{j}^{-1}),
\end{equation}
where
\begin{eqnarray}
a_{n}(1+\frac{\mu_{j}}{\lambda_{j}},\lambda_{j}^{-1})&=& \lambda_{j}\left(\frac{2n}{\pi} \lambda_{j}\right)^{1/4}\exp(-\sqrt{4\pi n\lambda_{j}})\ \cos\left(\sqrt{4\pi n\lambda_{j}}-\frac{5\pi}{8}-2\pi(\lambda_{j}+\mu_{j})\right)\nonumber\\
&&+O\left(n^{-1/4}e^{-2\sqrt{\pi n\lambda_{j}}}\right).\nonumber
\end{eqnarray}
The $a_{n}$ are exponentially small, then
\begin{equation}\label{eq4.7}
a_{n}(1+\frac{\mu_{j}}{\lambda_{j}},\lambda_{j}^{-1})=O(1).\end{equation}
From \eqref{eq4.6} and \eqref{eq4.7}, we obtain
\begin{equation}\label{eq4.8}
I_{j}=\left(\lambda_{j}+\mu_{j}-\frac{1}{2}\right)-n\lambda_{j}\left\{\psi(\lambda_{j}+\mu_{j})+\log(\lambda_{j}^{-1})+1-h_{n-1}\right\}+O(n).
\end{equation}
Summing \eqref{eq4.8} over $j$, we get
\begin{equation}\label{eq4.9}
\sum_{j=1}^{r}I_{j}=\sum_{j=1}^{r}\left(\lambda_{j}+\mu_{j}-\frac{1}{2}\right)-n\sum_{j=1}^{r}\lambda_{j}\left\{\psi(\lambda_{j}+\mu_{j})+\log(\lambda_{j}^{-1})+1-h_{n-1}\right\}+O(n).
\end{equation}
Using the expression
$$\psi(z)=-\gamma-\frac{1}{z}+\sum_{l=1}^{+\infty}\frac{z}{l(l+z)},$$
where $\gamma$ is the Euler's constant, and the estimate
$$h_{n}=\log n-\gamma+\frac{1}{2n}+O\left(\frac{1}{2n^{2}}\right),$$
we deduce from \eqref{eq4.5} and \eqref{eq4.9} that
\begin{eqnarray}
\lambda_{F}(-n)&=&\left(\sum_{j=1}^{r}\lambda_{j}\right)n\log n+\left\{\left(\sum_{j=1}^{r}\lambda_{j}\right)(\gamma-1)+\log Q_{F}+\sum_{j=1}^{r}\lambda_{j}\log\lambda_{j}\right)n\nonumber\\
&&-\sum_{l=1}^{n}(_{l}^{n})\eta_{F}(l-1)+ O(n).\nonumber
\end{eqnarray}
Recalling that $\displaystyle{d_{F}=\sum_{j=1}^{r}\lambda_{j}}$ and noting that $\displaystyle{\lambda=\prod_{j=1}^{r}\lambda_{j}^{2\lambda_{j}}}$, we have
\begin{equation}\label{eq4.10}
\lambda_{F}(-n)=\frac{d_{F}}{2}n\log n+\left\{\frac{d_{F}}{2}(\gamma-1)+\frac{1}{2}\log\left(\lambda Q_{F}^{2}\right)\right\}n-\sum_{l=1}^{n}(_{l}^{n})\eta_{F}(l-1)+ O(n).
\end{equation}
Now, we obtain a bound for $\displaystyle{S_{F}(n)=-\sum_{l=1}^{n}(_{l}^{n})\eta_{F}(l-1)}$ in terms of $$\lambda_{F}(-n,T):=\sum_{\rho;\ |\Im\rho|\leq T}1-\left(1-\frac{1}{\rho}\right)^{n},$$ where $T$ is a parameter.
\begin{lem} \label{lem4.4}If the Generalized Riemann Hypothesis holds for $F\in{{\mathcal S}}$, then
$$S_{F}(n)=O(\sqrt{n}\log n).$$
\end{lem}
\begin{proof} The proof is an analogous to the argument used by Lagarias in \cite{7}. We use a contour integral argument, and we introduce the kernel function
$$k_{n}:=\left(1+\frac{1}{s}\right)^{n}-1=\sum_{l=1}^{n}(^{n}_{l})\left(\frac{1}{s}\right)^{l}.$$
If $C$ is a contour enclosing the point $s=0$ counterclockwise on a circle of small enough positive radius, the residue theorem gives
$$I(n)=\frac{1}{2i\pi}\int_{C}k_{n}(s)\left(-\frac{F'}{F}(s+1)\right)ds=\sum_{l=1}^{n}(^{n}_{l})\eta_{l-1}=S_{F}(n).$$
We deform the contour to the counterclockwise oriented rectangular contour $C'$ consisting of vertical lines with real part $\Re(s)=\sigma_{0}$ and $\Re(s)=\sigma_{1}$, where we will choose $-3<\sigma_{0}<-2$, $\sigma_{1}=2\sqrt{n}$ and the horizontal lines at $\Im(s)=\pm T$, where we will choose $T=\sqrt{n}+\epsilon_{n}$, for some $0<\epsilon_{n}<1$. The residue theorem gives
\begin{eqnarray}
I'(n)&=&\frac{1}{2i\pi}\int_{C'}k_{n}(s)\left(-\frac{F'}{F}(s+1)\right)ds\nonumber\\
&=&S_{F}(n)+\sum_{\rho;\ |\Im\rho|\leq T}\left(1+\frac{1}{\rho-1}\right)^{n}-1+O(1).\nonumber
\end{eqnarray}
The term $O(1)$ evaluates the residues coming from the trivial
zeros of $F(s)$. Using the symmetry
$\rho\longmapsto1-\overline{\rho}$, we can write
$$\left(\frac{1-\overline{\rho}}{-\overline{\rho}}\right)^{n}-1=\left(\frac{\overline{\rho}-1}{\overline{\rho}}\right)^{n}-1.$$
Then
$$I'(n)=S_{F}(n)-\lambda_{F}(-n,T)+O(1).$$
We have $$|\lambda_{F}(-n,\sqrt{n})-\lambda_{F}(-n,T)|=O(\log n).$$
This follows from the observation that $|T-\sqrt{n}|<1$, that there are $O(\log n)$ zeros in an interval of length one at this height, and that for each zero $\rho=\beta+i\gamma$ with $\sqrt{n}\leq |\Im(\rho)|<\sqrt{n}+1$ there holds $$\left|\left(\frac{\rho-1}{\rho}\right)\right|\leq\left|1+\frac{1}{n}\right|^{n/2}\leq2.$$
We now choose the parameters $\sigma_{0}$ and $T$ appropriately to avoid poles of the integrand. We may choose $\sigma_{0}$ so that the contour avoids any trivial zero and $T=\sqrt{n}+\epsilon_{n}$ with $0\leq\epsilon_{n}\leq1$ so that the horizontal lines do not approach closer than $O(\log n)$ to any zero of $F(s)$. Recall from \cite{16} that for $-2<\Re(s)<2$ there holds
$$\frac{F'}{F}(s)=\sum_{\{\rho;\ |\Im(\rho-s)|<1\}}\frac{1}{s-\rho}+O\left(\log(Q_{F}(1+|s|))\right).$$
Then on the horizontal line in the interval $-2\leq\Re(s)\leq2$, we have
$$\left|\frac{F'}{F}(s+1)\right|=O(\log^{2}T).$$
The Euler product for $F(s)$ converge absolutely for $\Re(s)>1$,
hence the Dirichlet series for $\frac{F'}{F}(s)$ converge
absolutely for $\Re(s)>1$. More precisely for $\sigma=\Re(s)>1$
$$\left|\frac{F'}{F}\right|(\sigma)<\infty.$$
For $\sigma=\Re(s)>2$, we obtain the bound
$$\left|\frac{F'}{F}(s)\right|\leq\left|\frac{F'}{F}\right|(\sigma)\leq2^{-(\sigma-2)}.$$
Consider the integral $I'(n)$ on the vertical segment $(L_{1})$ having $\sigma_{1}=2\sqrt{n}$. We have
$$\left|(1-\frac{1}{s})^{n}-1\right|\leq\left(1+\frac{1}{\sigma_{1}}\right)^{n}+1\leq\left(1+\frac{1}{2\sqrt{n}}\right)^{n}\leq\exp(\sqrt{n}/2)<2^{\sqrt{n}}.$$
Then $$\left|\frac{F'}{F}(s)\right|\leq C_{0}2^{-2(\sqrt{n}+2)}.$$
Furthermore, the length of the contour is $O(\frac{n}{\log n})$, we obtain $|I'_{L_{1}}|=O(1)$. Let $s=\sigma+it$ be a point on one of the tow horizontal segment. We have $T\geq\sqrt{n}$, so that
$$\left|1+\frac{1}{s}\right|\leq1+\frac{\sigma+1}{\sigma^{2}+T^{2}}.$$
By hypothesis $T^{2}\geq n$, so for $-2\leq\sigma\leq2$, we have
$$|k_{n}(s)|\leq\left(1+\frac{3}{4+n}\right)^{n}+1=O(1)$$ and $$\left|\frac{F'}{F}(s)\right|=O(\log^{2}T)=O(\log^{2}n),$$
since we have chosen the ordinate $T$ to stay away from zeros of $F(s)$. We step across the interval $(L_{2})$ toward the right, in segments of length 1, starting from $\sigma=2$. Furthermore
$$\left|\frac{k_{n}(s+1)+1}{k_{n}(s)+1}\right|\leq\left(1+\frac{1}{T^{2}}\right)^{n}\leq e,$$
we obtain an upper bound for $|k_{n}(s)\frac{F'}{F}(s)|$ that decreases geometrically at each step. After $O(\log n)$ steps it becomes $O(1)$, and  the upper bound is
$$|I'_{L_{2},L_{4}}(n)|=O(\log^{2}n+\sqrt{n})=O(\sqrt{n}).$$
For the vertical segment $(L_{3})$ with $\Re(s)=\sigma_{0}$, we have $|k_{n}(s)|=O(1)$ and $|\frac{F'}{F}(s)|=O\left[Q_{F}(\log (|s|+1))\right]$. Since the segment $(L_{3})$ has length $O(\sqrt{n})$, we obtain
$$|I'_{L_{3}}|=O(\sqrt{n}\log n).$$
Totalling the above bounds  gives
$$S_{F}(n)=\lambda_{F}(-n,T)+O(\sqrt{n}\log n),$$
with $T=\sqrt{n}+\epsilon_{n}$. If the Generalized Riemann Hypothesis holds for $F(s)$, then we have
$\left|1-\frac{1}{\rho-1}\right|=1.$ Since each zero contributes a term of absolute value at most 2 to
$\lambda_{F}(-n,T)$, we obtain using the zero density estimate ($N_{F}(T)\sim Tlog T$)
$$\lambda_{F}(-n,T)=O(T\log T+1).$$
Therefore $\lambda_{F}(-n,\sqrt{n})=O(\sqrt{n}\log n)$, and  Lemma \ref{lem4.4} follows.
\end{proof}
Using Lemma \ref{lem4.4} and the expression \eqref{eq4.10} of $\lambda_{F}(-n)$ and $\lambda_{F}(-n)=\overline{\lambda_{F}(n)}$, we obtain
$$\lambda_{F}(n)=\frac{d_{F}}{2}n\log n+\left\{\frac{d_{F}}{2}(\gamma-1)+\frac{1}{2}\log(\lambda Q_{F}^{2})\right\}n+O(\sqrt{n}\log n),$$
which concludes the proof of  Theorem \ref{th4.1}.

\noindent{\bf Exemples :}\\- In the case of the Riemann zeta function, we have $d_{\zeta}=1,\ Q_{\zeta}=\pi^{-1/2}$, and $\lambda=1/2$. This proves again under the Riemann Hypothesis the asymptotic formula established by A. Voros in \cite[equation (17), p. 59]{17}.\\
- Also, in the case of the principal $L$-function $L(s,\pi)$ attached to an irreducible cuspidal unitary automorpohic representation of $GL(N)$, as in Rudnick and Sarnak \cite[\S 2]{14}, we have $D_{L}=N,\  Q_{L}=Q(\pi)\pi^{-N/2}$, and $\lambda=2^{-n}$. We find under the Generalized Riemann Hypothesis the asymptotic formula for $\lambda_{n}(\pi)$ established by Lagarias in \cite[equations (1.12) and (1.13), p. 4]{7}.\\

\noindent{\bf Acknowledgements}\\
I am grateful to Prof. Maciej Radziejewski  for his many valuable comments about the published paper  and   his request to clarify the proof  of some formulas (mainly equation \eqref{eq4.4}).


\begin{thebibliography}{00}
\bibitem[{\bf 1}]{1} E. Bombieri, D. A. Hejhal, {\it On the distribution of zeros of linear combinations of Euler products},  Duke. Math. J.  {\bf80} (3) (1995), 821--862.
\bibitem[\bf{2}]{2} E. Bombieri and J. C. Lagarias, {\it Complements to Li's criterion for the Riemann hypothesis,} J. Number theory {\bf
77} (2) (1999),  274--287.
\bibitem[\bf{3}]{3} M. Coffey, {\it Relations and positivity results for the derivatives if the Riemann $\xi$-function,} J. Comput. Appl. Math. {\bf 166} (2004), 525--534.
\bibitem[\bf{4}]{4} M. W. Coffey, {\it Toward verification of the Riemann hypothesis,} Math. Phys. Anal. Geom. {\bf 8} (2005), 211--255.
\bibitem[\bf{5}]{5} P. Flajolet and L. Vepstas, {\it On differences of zeta values,}  J. Comput. Appl. Math. {\bf 220} (1-2) (2008), 58--73.
\bibitem[\bf{6}]{6} J. Kaczorowski and A. Perelli, {\it On  the Selberg class : a survey,} Acta. Math. {\bf 192} (1999),  953--992.
\bibitem[\bf{7}]{7} J. C. Lagarias, {\it Li Coefficients for Automorphic L-Functions,} Annales Inst. Fourier {\bf 56} (2006), 1--52 .
\bibitem[\bf{8}]{8} X. -J. Li, {\it The positivity of a sequence of number and the Riemann hypothesis,} J. Number theory {\bf 65} (2) (1997),  325--333.
\bibitem[\bf{9}]{9} X. -J. Li, {\it Explicit formulas for Dirichlet and Hecke L-functions,} Illinois J. Math. vol 48. n {\bf 2} (2004),  491--503.
\bibitem[\bf{10}]{10} K. Maslanka, {\it Effective method of computing Li's coefficients and their properties,} Experimental Math (to appear) ( math.NT/0402168 v5 ).
\bibitem[\bf{11}]{11} S. Omar, K, Mazhouda, {\it Le critère de Li et l'hypoth\`ese de Riemann pour la classe de Selberg,}  Journal of Number Theory {\bf 125} (1) (2007), 50--58.
\bibitem[\bf{12}]{12} S. Omar, K, Mazhouda, {\it Le crit\`ere de positivit\' e de  Li pour la classe de Selberg,}  C. R. Acad. Sci. Paris, Série {\bf 345}, issue 5, 1 septembre (2007), 245--248.
\bibitem[\bf{13}]{13} S. Omar, K, Mazhouda, {\it Corrigendum et addendum `a Le crit`ere de Li et l'hypoth`ese de Riemann pour la classe de
Selberg [J. Number Theory 125(2007), no. 1, 50–58]}, J. Number Theory 130(2010), no. 4,
1109–1114.
\bibitem[{\bf 14}]{14} Z. Rudnick, P. Sarnak, {\it Zeros of principal $L$-functions and radom matrix theory}, Duke. Math. J. {\bf 81} (1996), 269-322.
\bibitem[\bf{15}]{15} A. Selberg, {\it Old and new conjectures and results about a class of Dirichlet series,} In: Proceedings of the
Amalfi conference on analytic number theory (Maiori, 1989), Univ. Salermo (1992), 367-–385.
 \bibitem[{\bf 16}]{16} K. Srinivas, {\it Distinct zeros of functions in the Selberg class}, Acta. Arith. {\bf 103}. 3 (2002), 201--207.
\bibitem[{\bf 17}]{17} A. Voros, {\it A sharpening of Li's criterion for the Riemann hypothesis}, Mathematical Physics Analysis and Geometry {\bf 9} (1) (2005), 53--63.
\end{thebibliography}
\end{document}